\documentclass[11pt]{article}

\usepackage[margin=1in]{geometry}
\usepackage{amsmath,amssymb,amsthm,mathrsfs}
\usepackage{microtype}
\usepackage{booktabs,array,tikz,placeins}
\usetikzlibrary{arrows.meta,positioning,fit}
\usepackage[colorlinks=true,citecolor=blue,linkcolor=blue,urlcolor=black]{hyperref}
\hypersetup{
  pdftitle={Zeros and interlacing for multiset Eulerian--Narayana polynomials},
  pdfauthor={Philip B. Zhang and Tongyuan Zhao}
}

\newtheorem{theorem}{Theorem}[section]
\newtheorem{proposition}[theorem]{Proposition}
\newtheorem{lemma}[theorem]{Lemma}
\newtheorem{corollary}[theorem]{Corollary}
\theoremstyle{remark}
\newtheorem{remark}[theorem]{Remark}
\newtheorem{example}[theorem]{Example}
\newcommand{\T}{\mathcal T}
\newcommand{\E}{\mathscr E}
\newcommand{\Int}{\mathsf{int}}
\newcommand{\Leaf}{\mathsf{leaf}}

\title{Zeros and interlacing for multiset Eulerian--Narayana polynomials}
\author{Philip B. Zhang$^{1}$ \qquad Tongyuan Zhao$^{2}$\\[1mm]
{\small $^{1}$College of Mathematical Sciences \& Institute of Mathematics and Interdisciplinary Sciences,}\\[-1mm]
{\small Tianjin Normal University, Tianjin 300387, P. R. China}\\[1mm]
{\small $^{2}$College of Science, China University of Petroleum (Beijing), Beijing 102249, P. R. China}\\[-1mm]
{\small \texttt{zhang@tjnu.edu.cn} \qquad \texttt{zhaotongyuan@cup.edu.cn}}}
\date{}

\begin{document}
\maketitle

\begin{abstract}
We prove the real-rootedness conjecture of Lin, Ma, Ma, and Zhou (2021)
for multiset Eulerian--Narayana polynomials, including
simplicity and negativity of all zeros after removing the factor $t$.
Using a forest decomposition and Lagrange inversion, we derive a
differential formula and a first-order factorization of the differential operator.
We use these factors to show that the nonconstant $\gamma$-coefficients
and ordered $\gamma$-zeros are nondecreasing, with corresponding
bounds given by the Narayana and Eulerian polynomials. The same factorization
gives totally nonnegative coefficient transition matrices. We also prove
strict interlacing of the zeros of generating polynomials arising from
adjacent columns of the transition matrices.
\end{abstract}



\section{Introduction}

Lin, Ma, Ma, and Zhou~\cite{LMMZ21} introduced weakly increasing
trees on multisets as a common extension of increasing trees and plane
trees. Their leaf enumerators interpolate between the Eulerian and
Narayana polynomials.

Let
\[
M=M(p_1,\ldots,p_m)=\{1^{p_1},\ldots,m^{p_m}\},
\qquad p_i\ge1,\qquad |M|=N,
\]
where $p_i$ is the multiplicity of label $i$. A weakly increasing
tree on $M$ is a plane tree on the multiset union $M\uplus\{0\}$,
rooted at $0$, with labels
weakly increasing along every root-to-leaf path and from left to
right among siblings. Write $\T_M$ for these trees and $\Leaf(T)$
and $\Int(T)$ for their numbers of leaves and internal nodes,
counting the root as internal for nonempty $M$.
Set
\[
A_M(t)=\sum_{T\in\T_M}t^{\Leaf(T)},
\qquad P_M(t)=A_M(t)/t.
\]
Lin, Ma, Ma, and Zhou proved a $\gamma$-positive expansion for
$P_M(t)$~\cite{LMMZ21}, which implies palindromicity. They also
conjectured that \emph{$A_M(t)$ has only real zeros for every multiset
$M$ with positive multiplicities}~\cite{LMMZ21}.
Since $\Leaf(T)+\Int(T)=N+1$, palindromicity gives
\[
P_M(t)=t^{N-1}P_M(1/t)
=\sum_{T\in\T_M}t^{N-\Leaf(T)}
=\sum_{T\in\T_M}t^{\Int(T)-1}.
\]
Thus the normalized leaf and internal-node enumerators agree.
The polynomial $P_M(t)$ is monic of degree $N-1$ with constant term $1$.

For $M=\{1,\ldots,m\}$ and $M=\{1^m\}$, we recover the Eulerian and
Narayana polynomials of degree $m-1$, respectively~\cite{LMMZ21}.
See Liu and Wang~\cite[Sections~3.4 and~3.6]{LiuWang07} for their
real-rootedness via interlacing. Lin, Ma, Ma, and Zhou~\cite{LMMZ21} identified the
cases $\{1^2,\ldots,m^2\}$ and
$\{1^2,\ldots,(m-1)^2,m\}$ with $\mathbf{s}$-Eulerian
polynomials, whose real-rootedness follows from Savage and
Visontai~\cite[Theorem~1.1]{SavageVisontai15}. For
$M=\{1^k,\ldots,(m-1)^k,m\}$, Lin, Xu, and
Zhao~\cite[Theorems~3.2 and~3.3]{LXZ22} showed that $P_M(t)$ coincides with the
descent polynomial for permutations of $\{1^k,\ldots,m^k\}$, so
Simion's theorem~\cite{Simion84} applies, and we refer
to~\cite[Section~4.2.1]{DHS26} for further discussion.
For arbitrary multiplicities, we prove the conjecture and show
that every zero of $A_M(t)$ is simple.

\begin{theorem}\label{thm:main-real-rooted}
Let $M$ be a nonempty multiset of size $N$. If $N=1$, then
$P_M(t)=1$. If $N\ge2$, all $N-1$ zeros of $P_M(t)$ are simple
and negative.
\end{theorem}

Deleting the largest label leaves a smaller tree with an ordered
forest attached at each vertex. A M\"obius substitution and
Lagrange--B\"urmann inversion turn the forest generating function
into a differential formula, to which Rolle's theorem applies.

Related work on weakly increasing trees includes bijections and
equidistributions~\cite{ChenFu22,LinLiuWangZang24} and multiset
Schett polynomials~\cite{LinMa25}.
The multiset Eulerian--Narayana polynomials studied here differ
from the usual multiset Eulerian polynomials, which enumerate
descents of multipermutations. Simion~\cite{Simion84}
proved real-rootedness of the latter for arbitrary multisets.
Ma and Pan~\cite{MaPan23} studied refinements by the first entry,
and Deligeorgaki, Han, and Solus~\cite{DHS26} considered colored
analogues and self-interlacing properties.

To compare multiplicities of fixed total size, we define $\Gamma_M(z)$ by
\[
P_M(t)=(1+t)^{N-1}
\Gamma_M\left(\frac{t}{(1+t)^2}\right).
\]
Theorem~\ref{thm:main-real-rooted} and palindromicity imply that
$\Gamma_M(z)$ has degree $q=\lfloor(N-1)/2\rfloor$ and simple
negative zeros. This also recovers the $\gamma$-positivity of $P_M(t)$
proved in~\cite{LMMZ21}, using the standard implication for palindromic
polynomials with nonnegative coefficients and real
zeros~\cite[Remark~7.3.1]{Branden15}.
When $q=0$, we have $\Gamma_M(z)=1$ and its zero list is empty.
Throughout, inequalities over empty index ranges are vacuous.
For a composition $\alpha=(p_1,\ldots,p_m)$ of $N$, we write
$P_\alpha(t)=P_{M(p_1,\ldots,p_m)}(t)$ and
$\Gamma_\alpha(z)=\Gamma_{M(p_1,\ldots,p_m)}(z)$.
We denote the zeros of $\Gamma_\alpha(z)$ by
\[
\xi_1(\alpha)<\cdots<\xi_q(\alpha)<0.
\]
A composition refines another if we can
obtain the former by splitting parts of the latter into consecutive
positive parts. Write $P_\alpha(t)\preceq_\gamma P_\beta(t)$ if every
coefficient of $\Gamma_\alpha(z)$ is at most the corresponding
coefficient of $\Gamma_\beta(z)$.

\begin{theorem}\label{thm:refinement}
Let $\beta$ refine $\alpha$. Then
\[
P_\alpha(t)\preceq_\gamma P_\beta(t),
\qquad \xi_i(\alpha)\le\xi_i(\beta)
\quad\text{for }1\le i\le q.
\]
If $P_\alpha(t)\ne P_\beta(t)$, every nonconstant $\gamma$-coefficient
inequality and every zero inequality is strict.
\end{theorem}

For example, the refinements $(1,3)$ and $(2,2)$ of $(4)$ give
\[
\Gamma_{(4)}(z)=\Gamma_{(1,3)}(z)=1+3z,
\qquad \Gamma_{(2,2)}(z)=1+5z.
\]
Let $\mathbf{1}_k=(1,\ldots,1)$ have $k$ parts, with
$\mathbf{1}_0$ empty. Every composition of $N$ lies between
$(N)$ and $\mathbf{1}_N$ in refinement order. The corresponding
Narayana and Eulerian polynomials are
\[
\mathrm{Nar}_N(t)=\frac1N\sum_{j=0}^{N-1}
\binom Nj\binom N{j+1}t^j,
\qquad
\mathrm{Eul}_N(t)=\sum_{j=0}^{N-1}
\left\langle\begin{matrix}N\\ j\end{matrix}\right\rangle t^j,
\]
where the Eulerian number counts permutations of $\{1,\ldots,N\}$
with $j$ descents. Theorem~\ref{thm:refinement} gives
\[
\mathrm{Nar}_N(t)\preceq_\gamma P_M(t)
\preceq_\gamma\mathrm{Eul}_N(t)
\]
and the analogous bounds on ordered $\gamma$-zeros.
Corollary~\ref{cor:sandwich} gives the equality cases and
ordinary coefficient bounds.

We next study the coefficient transition matrices for insertion of a
maximal-label block. A matrix is totally nonnegative if every square
minor is nonnegative, and totally positive if every square minor is positive.
The Aissen--Schoenberg--Whitney theorem~\cite{ASW52} characterizes
real-rootedness of a nonzero polynomial with nonnegative coefficients
by total nonnegativity of the infinite Toeplitz matrix of its coefficients.
Bidiagonal factorizations provide a standard method for proving total
nonnegativity~\cite[Chapter~2]{FallatJohnson11}, and
Brenti~\cite{Brenti95} discusses combinatorial applications of total positivity.

In our setting, the first-order factors act on coefficients by
bidiagonal matrices with positive entries on both diagonals. Their
product is the transition matrix in the recurrence
of Lin, Xu, and Zhao~\cite[Theorem~4.1]{LXZ22}, leading to our
last two results.
For a nonempty multiset $M$ of size $p$, write
\[
P_M(t)=\sum_{k=1}^{p}c_{M,k}t^{k-1},
\qquad c_{M,k}=\bigl|\{T\in\T_M:\Int(T)=k\}\bigr|.
\]
Adjoin $r\ge1$ copies of a new maximal label $a$, giving
$M^{(r)}=M\uplus\{a^r\}$. For a fixed $T\in\T_M$ with
$l$ internal nodes, let $K_{k,l}^{(p,r)}$ count trees in
$\T_{M^{(r)}}$ with $k$ internal nodes that reduce to $T$
after deletion of all $a$-nodes.
Section~\ref{sec:transitions} shows that this number depends
only on $p,r,k,l$. The resulting $(p+r)\times p$ matrix
$K^{(p,r)}$ satisfies
\[
c_{M^{(r)},k}=\sum_{l=1}^{p}c_{M,l}K_{k,l}^{(p,r)}.
\]

\begin{theorem}\label{thm:intro-TN}
For $p,r\ge1$, the matrix $K^{(p,r)}$ is totally nonnegative.
\end{theorem}

We count additional internal nodes with the $l$th column polynomial
\[
R_{p,l,r}(t)=\sum_{s=0}^{r}K_{l+s,l}^{(p,r)}t^s
\qquad(1\le l\le p).
\]
Thus $t^{l-1}R_{p,l,r}(t)$ is the generating polynomial of the $l$th
column in the monomial basis.
Throughout, interlacing includes the equal-degree case; it is strict
when all inequalities in the ordering of the zeros are strict.

\begin{theorem}\label{thm:intro-columns}
For $p,r\ge1$, each column polynomial $R_{p,l,r}(t)$ has $r$ simple
negative zeros, and adjacent column polynomials have strictly
interlacing zeros.
\end{theorem}

This paper is organized as follows.
In Section~\ref{sec:forest}, we derive the differential formula from
the forest decomposition, prove real-rootedness, and obtain the
first-order factorization. In Section~\ref{sec:refinement}, we use this
factorization to compare the $\gamma$-coefficients and zeros under
refinement and obtain bounds from the Narayana and Eulerian polynomials.
In Section~\ref{sec:transitions}, we study the coefficient transition
matrices for block insertion, determine their positive minors, and
prove strict interlacing of the zeros of adjacent column polynomials.
In Section~\ref{sec:final-remarks}, we discuss the limitations of these
comparisons and a related conjecture for multiset Schett polynomials.

\section{Forest decomposition and real-rootedness}\label{sec:forest}

In this section, we prove Theorem~\ref{thm:main-real-rooted} using a
differential formula for maximal-label insertion. We derive this formula by decomposing
a tree into a smaller tree and attached forests and then use Rolle's
theorem to locate the zeros. We also factor the differential operator
into first-order factors, which we use in the next two sections to
study refinement and coefficient transition matrices.

\subsection{Forest decomposition and recurrence}

Write $p=|M|$ for the size before insertion and $v=p+1$ for the
number of tree vertices. The identity $\Leaf(T)+\Int(T)=v$ and
palindromicity give, for nonempty $M$,
\begin{equation}\label{eq:leaf-internal}
A_M(t)=t^pP_M(1/t).
\end{equation}
We also allow the empty multiset $M=\varnothing$, whose tree is a
single root with $A_\varnothing(t)=t$. Let $a$ exceed every label of $M$, and put
$M^{(r)}=M\uplus\{a^r\}$ for $r\ge0$.
An $a$-tree is a plane tree in which every node carries the label $a$.
Let $\mathcal U(u,t)\in\mathbb Q[t][[u]]$ count nonempty $a$-trees,
with $u$ marking nodes and $t$ marking leaves.  The root is either a
leaf or has a nonempty ordered sequence of subtrees, so
\begin{equation}\label{eq:pure-tree-leaf}
\mathcal U(u,t)=u\left(t+\frac{\mathcal U(u,t)}{1-\mathcal U(u,t)}\right).
\end{equation}
The solution has zero constant term. Ordered forests have generating
function $\mathcal F(u,t)=(1-\mathcal U(u,t))^{-1}$, so
\begin{equation}\label{eq:forest-leaf}
\mathcal F(u,t)-1=u\mathcal F(u,t)\bigl(t+\mathcal F(u,t)-1\bigr).
\end{equation}

\begin{lemma}\label{lem:maximal-forest-decomposition}
For every $r\ge0$, deleting all nodes labeled $a$ gives a bijection between
$\T_{M^{(r)}}$ and pairs consisting of a base tree $T\in\T_M$ together with,
at each node of $T$, a possibly empty ordered forest of $a$-trees, with a total
of $r$ new nodes.
\end{lemma}

\begin{proof}
Delete the nodes labeled $a$ from $T'\in\T_{M^{(r)}}$.
A node with a label in $M$ cannot descend from an $a$-node, since labels
weakly increase along paths.  The remaining nodes therefore form a
connected tree $T\in\T_M$.  At each node of $T$, the deleted children
occupy the right end of the child list, and all their descendants carry
the label $a$, so they form an ordered forest of $a$-trees.

Conversely, attach any such forest to the right end of the child list
at each node of $T$.  Since $a$ exceeds every label in $M$, this
construction preserves both monotonicity conditions and reverses the
deletion.
\end{proof}

Figure~\ref{fig:maximal-forest-decomposition} shows both the deletion
and the forests attached to each vertex of the base tree.

% Requires \usepackage{tikz}; no additional TikZ libraries are needed.
% All coordinates are in centimetres.  The diagram is about 12 cm wide.
\begin{figure}[tbp]
\centering
\begin{tikzpicture}[
  x=1cm,y=1cm,
  every node/.style={font=\small},
  base/.style={circle,draw,fill=white,inner sep=0pt,minimum size=5.5mm},
  added/.style={base,fill=black!15},
  edge/.style={line width=.45pt},
  vertexname/.style={font=\scriptsize,inner sep=1pt}
]
  % The full tree.  Each child's horizontal position fixes its plane order.
  \node at (2.2,.85) {$T'$};
  \node[base] (t0) at (2.0,0) {$0$};
  \node[base] (t1) at (.9,-1.15) {$1$};
  \node[base] (t2) at (3.2,-1.15) {$1$};
  \node[base] (t3) at (.1,-2.3) {$2$};
  \node[added] (a0) at (4.4,-1.15) {$3$};
  \node[added] (a1) at (1.0,-2.3) {$3$};
  \node[added] (a2) at (1.9,-2.3) {$3$};
  \node[added] (a3) at (3.2,-2.3) {$3$};
  \node[added] (a4) at (3.2,-3.25) {$3$};
  \foreach \parent/\child in {t0/t1,t0/t2,t0/a0,t1/t3,t1/a1,t1/a2,t2/a3,a3/a4}
    \draw[edge] (\parent) -- (\child);
  \node[vertexname,above left] at (t0.north west) {$v_0$};
  \node[vertexname,left] at (t1.west) {$v_1$};
  \node[vertexname,right] at (t2.east) {$v_2$};
  \node[vertexname,left] at (t3.west) {$v_3$};

  \draw[->,line width=.55pt] (5.0,-1.15) -- (7.3,-1.15)
    node[midway,above,align=center] {delete all\\$3$-nodes};

  % The base tree; the vertex names distinguish the two nodes labelled 1.
  \node at (9.5,.85) {$T$};
  \node[base] (b0) at (9.5,0) {$0$};
  \node[base] (b1) at (8.7,-1.15) {$1$};
  \node[base] (b2) at (10.3,-1.15) {$1$};
  \node[base] (b3) at (8.7,-2.3) {$2$};
  \foreach \parent/\child in {b0/b1,b0/b2,b1/b3}
    \draw[edge] (\parent) -- (\child);
  \node[vertexname,above left] at (b0.north west) {$v_0$};
  \node[vertexname,left] at (b1.west) {$v_1$};
  \node[vertexname,right] at (b2.east) {$v_2$};
  \node[vertexname,left] at (b3.west) {$v_3$};

  % All four attached forests, including the empty forest, in vertex order.
  \node at (1.15,-4.05) {At $v_0$};
  \node at (4.2,-4.05) {At $v_1$};
  \node at (7.25,-4.05) {At $v_2$};
  \node at (10.3,-4.05) {At $v_3$};
  \node[added] at (1.15,-4.75) {$3$};
  \node[added] at (3.75,-4.75) {$3$};
  \node[added] at (4.65,-4.75) {$3$};
  \node[added] (c0) at (7.25,-4.75) {$3$};
  \node[added] (c1) at (7.25,-5.7) {$3$};
  \draw[edge] (c0) -- (c1);
  \node at (10.3,-4.75) {$\varnothing$};
\end{tikzpicture}
\caption{The decomposition in Lemma~\ref{lem:maximal-forest-decomposition}
for $M=\{1^2,2\}$ and $a=3$. The shaded nodes form the attached
forests shown below, including the empty forest at $v_3$. Vertex names
distinguish the two nodes labeled $1$.}
\label{fig:maximal-forest-decomposition}
\end{figure}
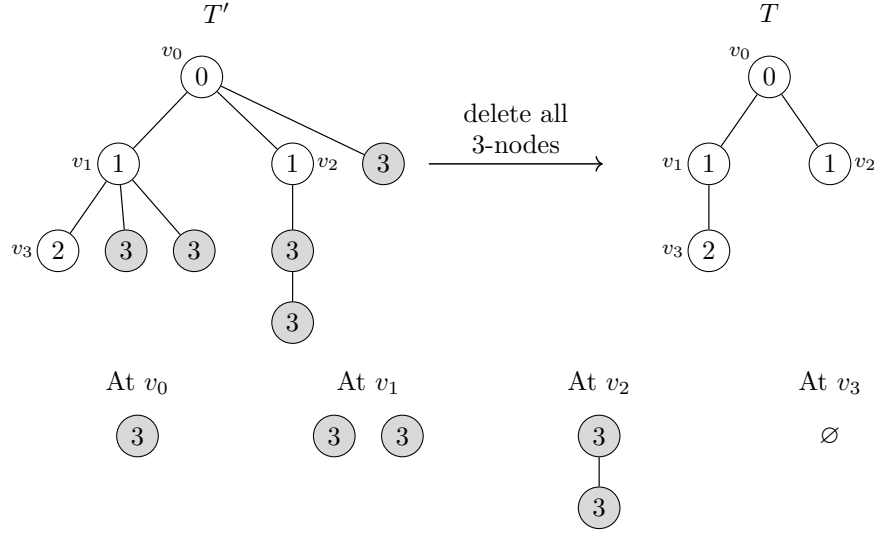

\begin{proposition}\label{prop:maximal-block}
For every finite multiset $M$,
\begin{equation}\label{eq:maximal-block}
\sum_{r\ge0}u^r A_{M^{(r)}}(t)
=\mathcal F(u,t)^v
A_M\left(1+\frac{t-1}{\mathcal F(u,t)}\right).
\end{equation}
\end{proposition}

\begin{proof}
By Lemma~\ref{lem:maximal-forest-decomposition}, we need only count
leaves after we attach the forests.  Suppose the base tree $T$ has
$\ell$ leaves. The variable $u$ counts only the new $a$-nodes,
so the nodes of the base tree contribute no power of $u$.
Each nonleaf contributes $\mathcal F(u,t)$. At a leaf,
the empty forest contributes $t$, while a nonempty forest makes the
base node internal and contributes $\mathcal F(u,t)-1$.  Thus the
extensions of $T$ have total weight
\[
\mathcal F(u,t)^{v-\ell}\bigl(t+\mathcal F(u,t)-1\bigr)^\ell
=\mathcal F(u,t)^v
\left(1+\frac{t-1}{\mathcal F(u,t)}\right)^\ell.
\]
Summing over $T\in\T_M$ proves~\eqref{eq:maximal-block}.
\end{proof}

\subsection{The differential formula and real-rootedness}

The following change of variables turns~\eqref{eq:maximal-block}
into a formula involving derivatives. Set
\[
\mu(x)=\frac{2x-1}{2x+1}.
\]
For any finite multiset $M$, put $v_M=|M|+1$ and define
\begin{equation}\label{eq:H-def}
H_M(x)=2^{-v_M}(2x+1)^{v_M}A_M(\mu(x)).
\end{equation}
This is a polynomial of degree $v_M$ with positive leading
coefficient $A_M(1)=|\T_M|$.

\begin{lemma}\label{lem:mobius-block}
Let $X=X(x,\eta)\in\mathbb Q(x)[[\eta]]$ be the unique formal solution of
\begin{equation}\label{eq:X-Lagrange}
X=x+\eta\left(X^2-\frac14\right),
\qquad X(x,0)=x.
\end{equation}
Then
\begin{equation}\label{eq:H-block-series}
\sum_{r\ge0}\eta^rH_{M^{(r)}}(x)=H_M(X).
\end{equation}
\end{lemma}

\begin{proof}
In~\eqref{eq:maximal-block}, put $t=\mu(x)$ and define $X$ and $\eta$ by
\[
2X+1=(2x+1)\mathcal F(u,\mu(x)),
\qquad
\eta=\frac{2u}{2x+1}.
\]
Then
\[
1+\frac{\mu(x)-1}{\mathcal F(u,\mu(x))}
=\mu(X).
\]
Using these substitutions in~\eqref{eq:forest-leaf}, we obtain
\eqref{eq:X-Lagrange}.
Multiplying~\eqref{eq:maximal-block} by $2^{-v}(2x+1)^v$ transforms its
right-hand side into $H_M(X)$.  Since
$u^r=\eta^r(2x+1)^r/2^r$, we use~\eqref{eq:H-def} with
$v_{M^{(r)}}=v+r$ to identify the $r$th term on the left as
$\eta^rH_{M^{(r)}}(x)$.  This proves~\eqref{eq:H-block-series}.
\end{proof}

Put
\[
Q(x)=x^2-\frac14.
\]
For a formal power series $f(u)$, we write $[u^k]f(u)$ for the
coefficient of $u^k$.  Lagrange--B\"urmann inversion now gives the
required differential formula.

\begin{proposition}\label{prop:block-differential}
For every $r\ge1$,
\begin{equation}\label{eq:block-differential}
H_{M^{(r)}}(x)
=\frac1{r!}
\left(\frac d{dx}\right)^{r-1}
\bigl(Q(x)^rH_M'(x)\bigr).
\end{equation}
\end{proposition}

\begin{proof}
Write $X=x+Z$.  Then~\eqref{eq:X-Lagrange} becomes
\[
Z=\eta Q(x+Z).
\]
By the Lagrange--B\"urmann formula~\cite[Appendix~A.6]{FlajoletSedgewick09},
for $r\ge1$,
\begin{align*}
[\eta^r]H_M(X)
&=\frac1r[w^{r-1}]H_M'(x+w)Q(x+w)^r\\
&=\frac1{r!}
\left(\frac d{dx}\right)^{r-1}
\bigl(Q(x)^rH_M'(x)\bigr).
\end{align*}
By~\eqref{eq:H-block-series}, the coefficient on the left is
$H_{M^{(r)}}(x)$, so this identity gives the required differential formula.
\end{proof}

For nonempty $M$ of size $p$, define the polynomial
\begin{equation}\label{eq:G-def}
G_M(x)
=2^{-(p-1)}(2x-1)^{p-1}
P_M\left(\frac{2x+1}{2x-1}\right).
\end{equation}
Substituting~\eqref{eq:leaf-internal} into~\eqref{eq:H-def} gives
\begin{equation}\label{eq:H-QG}
H_M(x)=Q(x)G_M(x).
\end{equation}
For $M=\varnothing$, $H_\varnothing(x)=x-\tfrac12$.
Since $P_M(t)$ is monic with constant term $1$, \eqref{eq:G-def} gives
$G_M(\tfrac12)=1$ and $G_M(-\tfrac12)=(-1)^{p-1}$.

\begin{lemma}\label{lem:rolle-multiplicities}
Let $V(x)$ be a nonconstant real polynomial with distinct real zeros
$u_1<\cdots<u_k$ of multiplicities $m_1,\ldots,m_k$, where
$\sum_i m_i=\deg V(x)$. Then $V'(x)$ has multiplicity $m_i-1$ at $u_i$
and one simple zero in each interval $(u_i,u_{i+1})$, and has no
other zeros.
\end{lemma}

\begin{proof}
Differentiation gives multiplicity $m_i-1$ at $u_i$, and Rolle's
theorem gives a zero in every intervening interval. These account for
$\sum_i(m_i-1)+(k-1)=\deg V(x)-1$ zeros counted with multiplicity,
so each intervening zero is simple and there are no further zeros.
\end{proof}

We now prove Theorem~\ref{thm:main-real-rooted}.

\begin{proof}[Proof of Theorem~\ref{thm:main-real-rooted}]
We prove by induction on the number of distinct labels that
$H_M(x)$ has simple zeros at $-\tfrac12$ and $\tfrac12$ and
$|M|-1$ simple zeros in between, accounting for its full degree $|M|+1$.

For one distinct label, $M=\{1^r\}$ with $r\ge1$.
For this initial block, Proposition~\ref{prop:block-differential} gives
$H_{\{1^r\}}(x)=(r!)^{-1}(d/dx)^{r-1}Q(x)^r$.
Lemma~\ref{lem:rolle-multiplicities}, applied $r-1$ times, gives
the required endpoint zeros and $r-1$ simple interior zeros.

Assume the assertion for every multiset with $k$ distinct labels,
and let $M$ be such a multiset, with $p=|M|$. Adjoining $r\ge1$
copies of a new maximal label gives $M^{(r)}$ with $k+1$ distinct
labels. By the induction hypothesis, $H_M'(x)$ has $p$ simple
zeros in $(-\tfrac12,\tfrac12)$ and
is nonzero at both endpoints.  For $0\le j\le r-1$, repeated application
of Lemma~\ref{lem:rolle-multiplicities} shows that
$(d/dx)^j\bigl(Q(x)^rH_M'(x)\bigr)$ has endpoint multiplicities $r-j$
and $p+j$ simple interior zeros.  At $j=r-1$, the endpoints are simple
and the interior count is $p+r-1$.  By~\eqref{eq:block-differential},
this proves the assertion for $H_{M^{(r)}}(x)$.

Consequently, $G_M(x)=H_M(x)/Q(x)$ has $p-1$ simple zeros, all in
$(-\tfrac12,\tfrac12)$.  The map $t=(2x+1)/(2x-1)$ sends this interval
bijectively to $(-\infty,0)$ and has nonzero derivative there.  The
prefactor in~\eqref{eq:G-def} is also nonzero on the interval, so the
corresponding zeros of $P_M(t)$ have the same multiplicities.  Since
$\deg P_M(t)=p-1$, these are all its zeros.
\end{proof}

For nonempty $M$, we use palindromicity to determine the degree
and zeros of $\Gamma_M(z)$.
Put $d=|M|-1$ and $q=\lfloor d/2\rfloor$. Palindromicity~\cite{LMMZ21}
pairs the zeros reciprocally, except for $-1$, which occurs precisely
when $d$ is odd.

Choose $\rho_j<-1$ from each of the $q$ reciprocal pairs, and put
$c_j=-\rho_j-\rho_j^{-1}-2>0$ and $\varepsilon=d-2q$.  Since $P_M(t)$
is monic, factorization into its reciprocal zero pairs gives
\[
P_M(t)=(1+t)^\varepsilon
\prod_{j=1}^q\bigl((1+t)^2+c_jt\bigr),
\qquad
\Gamma_M(z)=\prod_{j=1}^q(1+c_jz).
\]
The function $-u-u^{-1}-2$ is strictly decreasing on $(-\infty,-1)$.
Thus the $c_j$ are distinct and positive, proving the asserted
degree and simple negative zeros of $\Gamma_M(z)$.

\subsection{First-order factorization}

For the refinement and matrix results, we need the first-order
factors of the operator in Proposition~\ref{prop:block-differential}.
For $a>0$, define
\begin{equation*}
\mathcal A_a(g(x))
=Q(x)g'(x)+aQ'(x)g(x),
\end{equation*}
where $Q(x)=x^2-\tfrac14$.  For every positive integer $a$,
\begin{equation}\label{eq:QA-derivative}
\frac d{dx}\bigl(Q(x)^ag(x)\bigr)
=Q(x)^{a-1}\mathcal A_a(g(x)).
\end{equation}

\begin{proposition}\label{prop:differential-factorization}
For every polynomial $g(x)$ and every $r\ge1$,
\begin{align*}
&\frac1{r!}\left(\frac d{dx}\right)^{r-1}
\left(Q(x)^r\frac d{dx}\bigl(Q(x)g(x)\bigr)\right)
\\
&\qquad=\frac{Q(x)}{r!}
\bigl(\mathcal A_2\circ\mathcal A_3\circ\cdots\circ
\mathcal A_r\circ\mathcal A_1\bigr)(g(x)),
\end{align*}
where we omit $\mathcal A_2\circ\cdots\circ\mathcal A_r$ when $r=1$.
\end{proposition}

\begin{proof}
The inner derivative is $\mathcal A_1(g(x))$. By~\eqref{eq:QA-derivative},
each subsequent derivative lowers the power of $Q(x)$ by one and applies
$\mathcal A_r,\mathcal A_{r-1},\ldots,\mathcal A_2$, in that order.
After $r-1$ derivatives, division by $r!$ gives the identity.
\end{proof}

For the first nonempty block, $G_{\{1\}}(x)=1$.  If $r\ge2$, Proposition~\ref{prop:block-differential},
$H_\varnothing'(x)=1$, and repeated use of~\eqref{eq:QA-derivative} give
\[
G_{\{1^r\}}(x)
=\frac1{r!}
\bigl(\mathcal A_2\circ\mathcal A_3\circ\cdots\circ
\mathcal A_r\bigr)(1).
\]
If $M$ is nonempty and $r\ge1$,
then~\eqref{eq:H-QG}, Proposition~\ref{prop:block-differential}, and
Proposition~\ref{prop:differential-factorization} give
\begin{equation}\label{eq:G-block-factorization}
G_{M^{(r)}}(x)
=\frac1{r!}
\bigl(\mathcal A_2\circ\mathcal A_3\circ\cdots\circ
\mathcal A_r\circ\mathcal A_1\bigr)(G_M(x)).
\end{equation}

To return to the original variable, let $d\ge0$ be an integer
and $a>0$ a real number, and define
\[
\mathscr C_d(f(t))(x)
=2^{-d}(2x-1)^d
f\left(\frac{2x+1}{2x-1}\right)
\]
and
\begin{equation}\label{eq:E-def}
\E_{d,a}(f(t))
=\frac1a\left(
 t(1-t)f'(t)+\bigl(a+(d+a)t\bigr)f(t)
\right).
\end{equation}
Then $G_M(x)=\mathscr C_{p-1}(P_M(t))(x)$ for nonempty $M$.

\begin{lemma}\label{lem:A-E-conjugacy}
For every polynomial $f(t)$ of degree at most $d$ and every $a>0$,
\begin{equation*}
\mathcal A_a\bigl(\mathscr C_d(f(t))(x)\bigr)
=a\,\mathscr C_{d+1}\bigl(\E_{d,a}(f(t))\bigr)(x).
\end{equation*}
\end{lemma}

\begin{proof}
Put $t=(2x+1)/(2x-1)$.  Then
\[
\frac{dt}{dx}=-(t-1)^2,
\qquad
x^2-\frac14=\frac{t}{(t-1)^2},
\qquad
\mathscr C_d(f(t))(x)=(t-1)^{-d}f(t).
\]
Differentiating the last expression and using
$Q'(x)=(t+1)/(t-1)$ gives
\[
\mathcal A_a\bigl(\mathscr C_d(f(t))(x)\bigr)
=(t-1)^{-d-1}
\bigl(t(1-t)f'(t)+\bigl(a+(d+a)t\bigr)f(t)\bigr).
\]
By~\eqref{eq:E-def}, this is
$a\,\mathscr C_{d+1}\bigl(\E_{d,a}(f(t))\bigr)(x)$.
\end{proof}

By injectivity of $\mathscr C_d$ and Lemma~\ref{lem:A-E-conjugacy},
the scalar factors $1,r,r-1,\ldots,2$ cancel the prefactor $1/r!$ in
\eqref{eq:G-block-factorization}. Hence, for nonempty $M$ of size $p$,
when $r=1$,
\begin{equation}\label{eq:block-univ-r1}
P_{M^{(1)}}(t)=\E_{p-1,1}(P_M(t)),
\end{equation}
whereas for $r\ge2$,
\begin{equation}\label{eq:block-univ}
P_{M^{(r)}}(t)
=\bigl(
\E_{p+r-2,2}\circ\E_{p+r-3,3}\circ\cdots\circ
\E_{p,r}\circ\E_{p-1,1}
\bigr)(P_M(t)).
\end{equation}
For the initial multiplicity block,
\begin{equation}\label{eq:initial-block}
P_{\{1^r\}}(t)
=\bigl(\E_{r-2,2}\circ\E_{r-3,3}\circ\cdots\circ
\E_{0,r}\bigr)(1),\qquad r\ge2,
\end{equation}
and $P_{\{1\}}(t)=1$.

\subsection{Interlacing along the factorization}

We allow a degree parameter larger than the input degree,
as needed for the column polynomials in Section~\ref{sec:transitions}.

\begin{lemma}\label{lem:E-root-insertion}
Let $h(t)$ have degree $q$, positive leading coefficient, and only simple
negative zeros.  Let $m$ and $d_*$ be nonnegative integers, let $a>0$, and suppose
$d_*\ge m+q$.  Define
\begin{equation*}
\widetilde h(t)=t(1-t)h'(t)+\bigl(m+a+(d_*+a-m)t\bigr)h(t).
\end{equation*}
Then $\widetilde h(t)$ has degree $q+1$, positive leading coefficient,
and simple negative zeros $\sigma_0<\cdots<\sigma_q$.
For $q\ge1$, these zeros and the zeros
$\rho_1<\cdots<\rho_q<0$ of $h(t)$ satisfy
\[
\sigma_0<\rho_1<\sigma_1<\cdots<\rho_q<\sigma_q<0.
\]
\end{lemma}

\begin{proof}
Let $\kappa>0$ be the leading coefficient of $h(t)$.  Since all zeros of
$h(t)$ are negative, $h(0)>0$.  The constant term of $\widetilde h(t)$ is
$(m+a)h(0)>0$, while its leading coefficient is
\[
(d_*+a-m-q)\kappa>0.
\]
For $q=0$, these positive coefficients give a linear polynomial with one
negative zero.  For $q\ge1$, at a zero
$\rho_i$ of $h(t)$,
\[
\widetilde h(\rho_i)=\rho_i(1-\rho_i)h'(\rho_i),
\qquad
\operatorname{sgn}\widetilde h(\rho_i)=(-1)^{q-i+1}.
\]
The sign at $-\infty$ is $(-1)^{q+1}$ and $\widetilde h(0)>0$.
Hence there is one zero in each of
$(-\infty,\rho_1)$, $(\rho_i,\rho_{i+1})$, and $(\rho_q,0)$.
Degree counting shows that these are all the zeros and that they are simple.
\end{proof}

For $f(t)=t^m h(t)$, we have
\begin{equation}\label{eq:E-root-action}
a\,\E_{d_*,a}(f(t))=t^m\widetilde h(t).
\end{equation}
After we remove the factor $t^m$ from $\E_{d_*,a}(f(t))$,
Lemma~\ref{lem:E-root-insertion} gives the remaining zeros and their order.

We combine the block formulas by indexing the factors according
to the degree at which they act.  For a composition $\alpha=(p_1,\ldots,p_m)$ of $N$, write
the partial sums
\[
N_j=p_1+\cdots+p_j\quad(1\le j\le m),\qquad N_m=N.
\]
For $0\le d\le N-2$, define the positive integer
\begin{equation}\label{eq:ad-prefix}
a_d(\alpha)=N_j-d,
\qquad
j=\min\{i:N_i>d\}.
\end{equation}

\begin{proposition}\label{prop:global-factorization}
If $N=1$, then $P_\alpha(t)=1$.  If $N\ge2$, then
\begin{equation}\label{eq:global-factorization}
P_\alpha(t)
=\bigl(
\E_{N-2,a_{N-2}(\alpha)}\circ\cdots\circ
\E_{1,a_1(\alpha)}\circ\E_{0,a_0(\alpha)}
\bigr)(1).
\end{equation}
\end{proposition}

\begin{proof}
The initial block contributes no factor if $p_1=1$. Otherwise,
\eqref{eq:initial-block} acts at degrees $0\le d\le p_1-2$ with
parameter $p_1-d=a_d(\alpha)$.

For a later block $p_j=r$, put $p=N_{j-1}$. Its first factor is
$\E_{p-1,1}$, and $a_{p-1}(\alpha)=N_{j-1}-(p-1)=1$.
When $r\ge2$, the remaining factors in~\eqref{eq:block-univ} act
at $p\le d\le p+r-2$ with parameter $p+r-d=N_j-d=a_d(\alpha)$.
Concatenating these blocks proves~\eqref{eq:global-factorization}.
\end{proof}

For $N\ge2$ and $0\le d\le N-2$, set
\[
P_\alpha^{[d]}(t)
=\bigl(
\E_{d,a_d(\alpha)}\circ\cdots\circ
\E_{1,a_1(\alpha)}\circ\E_{0,a_0(\alpha)}
\bigr)(1).
\]
Thus $P_\alpha^{[N-2]}(t)=P_\alpha(t)$.

Starting with $P_\alpha^{[-1]}(t)=1$, successive applications of
Lemma~\ref{lem:E-root-insertion} with $m=0$, $d_*=q=d$, and
$a=a_d(\alpha)$ show that $P_\alpha^{[d]}(t)$ has degree $d+1$ and
only simple negative zeros.  For $1\le d\le N-2$, the zeros of
$P_\alpha^{[d-1]}(t)$ strictly interlace those of $P_\alpha^{[d]}(t)$.

\begin{remark}\label{rem:factor-versus-block}
The intermediate polynomials need not arise from adjoining
whole multiplicity blocks.  For example, when $\alpha=(4)$, the
partial polynomial $P_\alpha^{[1]}(t)=1+\tfrac83t+t^2$ has nonintegral
coefficients.  Moreover, the polynomials before and after insertion of a whole
multiplicity block can share a zero.  For example,
$P_{(2)}(t)=1+t$ and
$P_{(2,2)}(t)=1+8t+8t^2+t^3$ both vanish at $-1$.
\end{remark}

\section{\texorpdfstring{$\gamma$}{gamma}-polynomials and refinement}\label{sec:refinement}

In this section, we use the factorization from Section~\ref{sec:forest}
to compare compositions of the same total size. We express its first-order
factors as operators on the $\gamma$-polynomials and track the changes
in their coefficients and zeros under refinement. We use these
comparisons to prove Theorem~\ref{thm:refinement}, classify the equality
cases, and obtain bounds from the Narayana and Eulerian polynomials.

\subsection{The \texorpdfstring{$\gamma$}{gamma}-operator and interlacing}

We first show that each first-order factor preserves palindromicity,
which allows us to express its action in the $\gamma$-basis.

\begin{lemma}\label{lem:palindromicity-preservation}
Let $f(t)$ be palindromic of degree $d$, that is,
\[
f(t)=t^df(1/t),
\]
and let $a>0$.  Then $\E_{d,a}(f(t))$ is palindromic of degree $d+1$.
\end{lemma}

\begin{proof}
Differentiating the palindromic identity gives
\[
f'(1/t)=d t^{1-d}f(t)-t^{2-d}f'(t).
\]
If $g(t)=a\E_{d,a}(f(t))$, then substitution into~\eqref{eq:E-def} yields
\[
t^{d+1}g(1/t)
=t(1-t)f'(t)+\bigl(a+(d+a)t\bigr)f(t)=g(t).
\]
The coefficient of $t^{d+1}$ in $g(t)$ is $a$ times the leading
coefficient of $f(t)$, proving the degree assertion.
\end{proof}

The factorization~\eqref{eq:global-factorization} thus recovers the
palindromicity of $P_M(t)$ proved in~\cite{LMMZ21}.
We write $z=t/(1+t)^2$ and express the two polynomials as
\[
f(t)=(1+t)^d\Gamma(z),
\qquad
\E_{d,a}(f(t))=(1+t)^{d+1}\widehat\Gamma(z).
\]
Using $dz/dt=(1-t)/(1+t)^3$ in~\eqref{eq:E-def} and
dividing by $(1+t)^{d+1}$, we obtain
\begin{equation}\label{eq:gamma-op}
\widehat\Gamma(z)
=\frac1a\left(
 z(1-4z)\Gamma'(z)+(a+2dz)\Gamma(z)
\right).
\end{equation}
Define the operators
\[
\Phi_{d,a}(\Gamma(z))
=\Gamma(z)+\frac1a\mathcal D_d(\Gamma(z)),
\qquad
\mathcal D_d(\Gamma(z))
=z(1-4z)\Gamma'(z)+2dz\Gamma(z).
\]
Writing
\[
\Gamma(z)=\sum_{i=0}^{\lfloor d/2\rfloor}\gamma_i z^i,
\qquad
\widehat\Gamma(z)=\sum_{i=0}^{\lfloor(d+1)/2\rfloor}
\widehat\gamma_i z^i,
\]
we obtain
\begin{equation}\label{eq:gamma-coeff-op}
\widehat\gamma_i
=\frac{a+i}{a}\gamma_i
+\frac{2(d-2i+2)}{a}\gamma_{i-1}.
\end{equation}
Here $\gamma_i=0$ for $i<0$ or $i>\lfloor d/2\rfloor$.
For $0\le i\le\lfloor(d+1)/2\rfloor$, the multipliers
in~\eqref{eq:gamma-coeff-op} are positive. Hence $\Phi_{d,a}$
preserves nonnegative coefficients, and
\begin{equation}\label{eq:Dd-positive}
[z^i]\mathcal D_d(\Gamma(z))
=i\gamma_i+2(d-2i+2)\gamma_{i-1}\ge0.
\end{equation}
The higher coefficients vanish, so $\mathcal D_d$ also preserves
nonnegative coefficients. Along~\eqref{eq:global-factorization}, all
intermediate $\gamma$-coefficients are positive, since the first term in
\eqref{eq:gamma-coeff-op} preserves positivity, and each new leading
coefficient is $2/a$ times the preceding one.

\begin{lemma}\label{lem:gamma-root-insertion}
Let $d\ge0$ be an integer, $a>0$, and $m=\lfloor d/2\rfloor$.
Suppose $\Gamma(z)$ has degree $m$, positive leading coefficient, and
only simple negative zeros $\xi_1<\cdots<\xi_m$.
Then $\widehat\Gamma(z)=\Phi_{d,a}(\Gamma(z))$ has degree
$\lfloor(d+1)/2\rfloor$, positive leading coefficient, and only simple
negative zeros.  For $m\ge1$, the zeros $\eta_i$ of $\widehat\Gamma(z)$
satisfy
\[
\begin{array}{ll}
\xi_1<\eta_1<\xi_2<\eta_2<\cdots<\xi_m<\eta_m<0,
& d=2m,\\[2pt]
\eta_1<\xi_1<\eta_2<\cdots<\xi_m<\eta_{m+1}<0,
& d=2m+1.
\end{array}
\]
\end{lemma}

\begin{proof}
When $m=0$, formula~\eqref{eq:gamma-op} gives
$\widehat\Gamma(z)=\Gamma(z)$ for $d=0$ and
$\widehat\Gamma(z)=\Gamma(z)(1+2z/a)$ for $d=1$, whose only zero is
$-a/2$. For $m\ge1$, all coefficients of $\Gamma(z)$ are positive, and
formula~\eqref{eq:gamma-coeff-op} shows that $\widehat\Gamma(z)$ has
positive leading coefficient and degree $\lfloor(d+1)/2\rfloor$.
When $d=2m$, the coefficient of $z^{m+1}$ vanishes, whereas for
$d=2m+1$ it equals $2\gamma_m/a>0$.

At a zero $\xi_j$ of $\Gamma(z)$,
$\widehat\Gamma(\xi_j)=a^{-1}\xi_j(1-4\xi_j)\Gamma'(\xi_j)\ne0$.
Thus $\Gamma(z)$ and $\widehat\Gamma(z)$ have no common zero. To locate
the zeros of $\widehat\Gamma(z)$, put
$\boldsymbol\xi=(\xi_1,\ldots,\xi_m)$ and define, for
$z\notin\{\xi_1,\ldots,\xi_m\}$,
\begin{equation}\label{eq:gamma-root-equilibrium}
\mathcal W(z,\boldsymbol\xi,a)
=\frac{a\widehat\Gamma(z)}{\Gamma(z)}
=a+2dz+z(1-4z)\sum_{j=1}^m\frac1{z-\xi_j}.
\end{equation}
Polynomial division gives
\begin{equation*}
\mathcal W(z,\boldsymbol\xi,a)
=a+2(d-2m)z+
\sum_{j=1}^m\left(1-4\xi_j+
\frac{\xi_j(1-4\xi_j)}{z-\xi_j}\right).
\end{equation*}
Consequently, on every pole-free real interval,
\begin{equation}\label{eq:gamma-equilibrium-derivative}
\mathcal W_z(z,\boldsymbol\xi,a)
=2(d-2m)+\sum_{j=1}^m
\frac{-\xi_j(1-4\xi_j)}{(z-\xi_j)^2}>0.
\end{equation}
Since every residue is negative, $\mathcal W(z,\boldsymbol\xi,a)$ tends to
$+\infty$ on the left of each pole and to $-\infty$ on its right.
There is exactly one zero in each $(\xi_i,\xi_{i+1})$, and, since
$\mathcal W(0,\boldsymbol\xi,a)=a>0$, exactly one in $(\xi_m,0)$.
For $d=2m$, the limit at $-\infty$ is
$a+\sum_{j=1}^m(1-4\xi_j)>0$, so there is no zero in $(-\infty,\xi_1)$.
For $d=2m+1$, the limit is $-\infty$, giving one additional zero there.
These zeros exhaust the degree of $\widehat\Gamma(z)$ and are
simple by~\eqref{eq:gamma-equilibrium-derivative}.
\end{proof}

For a composition $\alpha$ of $N$, let $\Gamma_\alpha^{[d]}(z)$ be the
$\gamma$-polynomial of $P_\alpha^{[d]}(t)$ for $0\le d\le N-2$, and set
$\Gamma_\alpha^{[-1]}(z)=1$.  Then
\begin{equation}\label{eq:gamma-prefix-rec}
\Gamma_\alpha^{[d]}(z)
=\Phi_{d,a_d(\alpha)}\bigl(\Gamma_\alpha^{[d-1]}(z)\bigr).
\end{equation}
Lemma~\ref{lem:gamma-root-insertion} gives strict interlacing at
successive steps, both when the degrees differ and when they are equal
and positive.

\begin{example}\label{ex:factor-chain-22}
For $\alpha=(2,2)$, the parameters in~\eqref{eq:ad-prefix} are
$(a_0,a_1,a_2)=(2,1,2)$, and hence
\[
P_{(2,2)}(t)=\E_{2,2}\E_{1,1}\E_{0,2}(1).
\]
Starting with $P_{(2,2)}^{[-1]}(t)=1$, the successive polynomials are
\[
\begin{array}{c|c|c}
d&P_{(2,2)}^{[d]}(t)&\Gamma_{(2,2)}^{[d]}(z)\\ \hline
0&1+t&1\\
1&1+4t+t^2&1+2z\\
2&1+8t+8t^2+t^3&1+5z
\end{array}
\]
Here $P_{(2,2)}^{[d]}(t)=(1+t)^{d+1}
\Gamma_{(2,2)}^{[d]}(t/(1+t)^2)$.
The first zero, $-1$, lies between the two zeros $-2\pm\sqrt3$
at the next step. The zeros in the last two steps satisfy
\[
\frac{-7-3\sqrt5}{2}<-2-\sqrt3<-1
<-2+\sqrt3<\frac{-7+3\sqrt5}{2}<0.
\]
The corresponding $\gamma$-zeros are $-\tfrac12<-\tfrac15$.
\end{example}

\subsection{Monotonicity of the zeros}

We use~\eqref{eq:gamma-root-equilibrium} to compare the transformed
zeros. Fix an integer $d\ge1$, and put
$m=\lfloor d/2\rfloor$ and
$q=\lfloor(d+1)/2\rfloor$.  Let
\[
\Gamma_{\boldsymbol\xi}(z)=\prod_{j=1}^m(z-\xi_j),
\qquad
\xi_1<\cdots<\xi_m<0,
\]
and for $a>0$ let
\[
\eta_1(\boldsymbol\xi,a)<\cdots<
\eta_q(\boldsymbol\xi,a)<0
\]
be the zeros we obtain by applying~\eqref{eq:gamma-op} to
$\Gamma_{\boldsymbol\xi}(z)$.
For another vector $\widetilde{\boldsymbol\xi}$ of length $m$ with
strictly increasing negative entries, the notation
$\boldsymbol\xi\le\widetilde{\boldsymbol\xi}$ means that
$\xi_j\le\widetilde\xi_j$ for every $j$.

\begin{proposition}\label{prop:gamma-root-monotonicity}
Each $\eta_i(\boldsymbol\xi,a)$ is strictly increasing in every
$\xi_j$ and strictly decreasing in $a$.  If
$\boldsymbol\xi\le\widetilde{\boldsymbol\xi}$ and $0<b\le a$, then
\[
\eta_i(\boldsymbol\xi,a)
\le
\eta_i(\widetilde{\boldsymbol\xi},b)
\]
for $1\le i\le q$, with strict inequalities for every $i$ unless
$\boldsymbol\xi=\widetilde{\boldsymbol\xi}$ and $a=b$.
\end{proposition}

\begin{proof}
For $m=0$, the zero $-a/2$ has the required monotonicity.
We therefore assume $m\ge1$ and use
$\mathcal W(z,\boldsymbol\xi,a)$ from~\eqref{eq:gamma-root-equilibrium}.
Every $\eta_i$ is negative and distinct from all $\xi_j$ by
Lemma~\ref{lem:gamma-root-insertion}.  At such a zero $\eta_i$,
\[
\mathcal W_{\xi_j}(\eta_i,\boldsymbol\xi,a)
=\frac{\eta_i(1-4\eta_i)}{(\eta_i-\xi_j)^2}<0,
\qquad \mathcal W_a(\eta_i,\boldsymbol\xi,a)=1.
\]
By~\eqref{eq:gamma-equilibrium-derivative}, $\mathcal W_z(\eta_i)>0$, where we
suppress the fixed parameters.  Implicit differentiation therefore gives
\[
\frac{\partial\eta_i}{\partial\xi_j}
=\frac{-\eta_i(1-4\eta_i)}{(\eta_i-\xi_j)^2\mathcal W_z(\eta_i)}>0,
\qquad
\frac{\partial\eta_i}{\partial a}=-\frac1{\mathcal W_z(\eta_i)}<0.
\]
To prove the final assertion, we interpolate both the zeros and the parameter by
\[
\xi_j(s)=(1-s)\xi_j+s\widetilde\xi_j,
\qquad
a(s)=(1-s)a+sb
\qquad(0\le s\le1).
\]
The interpolated vector remains strictly ordered and negative.
Lemma~\ref{lem:gamma-root-insertion} ensures that the transformed
zeros remain simple, so they form smooth branches satisfying
\[
\frac{d\eta_i}{ds}
=
\sum_{j=1}^m
\frac{\partial\eta_i}{\partial\xi_j}
(\widetilde\xi_j-\xi_j)
+
\frac{\partial\eta_i}{\partial a}(b-a)\ge0,
\]
with strict inequality if $\xi_j<\widetilde\xi_j$ for some $j$ or $b<a$.
Integrating from $s=0$ to $s=1$ gives the claimed comparison, including
strictness whenever at least one parameter changes.
\end{proof}

\subsection{Proof of the refinement theorem}

To prove Theorem~\ref{thm:refinement}, we first compare the parameters
of the first-order factors under refinement.
For a composition $\delta$, let $\operatorname{Cut}(\delta)$ denote
its set of proper prefix sums.

\begin{lemma}\label{lem:refinement-parameters}
Let $\beta$ be a refinement of a composition $\alpha$ of $N\ge2$.
Then
\[
a_d(\beta)\le a_d(\alpha)
\]
for $0\le d\le N-2$.
If $\operatorname{Cut}(\beta)\setminus\operatorname{Cut}(\alpha)
\subseteq\{1\}$, the parameters agree for $1\le d\le N-2$.
Otherwise, every new cut $c>1$ gives an index $d_0=c-1$ satisfying
\[
1\le d_0\le N-2,
\qquad a_{d_0}(\beta)=1<a_{d_0}(\alpha).
\]
\end{lemma}

\begin{proof}
Every prefix sum of $\alpha$ is a prefix sum of $\beta$.
The least prefix sum exceeding $d$ can therefore only decrease under
refinement, proving the first assertion by~\eqref{eq:ad-prefix}.
A new cut at $1$ does not change this least prefix sum when $d\ge1$.
If $c>1$ is a new cut, then the least prefix sum of $\beta$ exceeding
$c-1$ is $c$, whereas the corresponding prefix sum of $\alpha$ is
strictly greater than $c$. Since $c$ is proper, $c\le N-1$, giving
the stated range of $d_0$.
\end{proof}

For a composition $\alpha$ of $N$, we put $q=\lfloor(N-1)/2\rfloor$ and write
\[
\Gamma_\alpha(z)=\sum_{i=0}^{q}\gamma_{\alpha,i}z^i,
\qquad \gamma_{\alpha,0}=1.
\]

\begin{lemma}\label{lem:refinement-coefficients}
Let $\beta$ refine a composition $\alpha$ of $N\ge1$.
Then $P_\alpha(t)\preceq_\gamma P_\beta(t)$ and
$\gamma_{\alpha,0}=\gamma_{\beta,0}=1$.
The equality $P_\alpha(t)=P_\beta(t)$ holds if and only if
\[
\operatorname{Cut}(\beta)\setminus\operatorname{Cut}(\alpha)
\subseteq\{1\}.
\]
In every other case,
\[
\gamma_{\beta,i}>\gamma_{\alpha,i}
\]
for $1\le i\le\lfloor(N-1)/2\rfloor$.
\end{lemma}

\begin{proof}
The assertions are immediate when $N=1$, so assume $N\ge2$.

\emph{Coefficient monotonicity.}
For $0\le d\le N-2$, put $a=a_d(\alpha)$, $b=a_d(\beta)$, and
$\Delta_d(z)=\Gamma_\beta^{[d]}(z)-\Gamma_\alpha^{[d]}(z)$, with
$\Delta_{-1}(z)=0$.
Lemma~\ref{lem:refinement-parameters} gives $b\le a$, and
recurrence~\eqref{eq:gamma-prefix-rec} expresses the difference as
\begin{equation}\label{eq:refinement-difference}
\Delta_d(z)
=\Phi_{d,b}\bigl(\Delta_{d-1}(z)\bigr)
+\left(\frac1b-\frac1a\right)
\mathcal D_d\bigl(\Gamma_\alpha^{[d-1]}(z)\bigr).
\end{equation}
By~\eqref{eq:gamma-coeff-op} and~\eqref{eq:Dd-positive}, both terms
on the right are coefficientwise nonnegative whenever $\Delta_{d-1}(z)$
is. Induction proves the weak coefficient inequalities. Every $\Phi_{d,a}$ preserves the constant coefficient,
so both final constant coefficients equal $1$.

\emph{Equality and propagation of strictness.}
If the only possible new cut is $1$, the parameters agree for every
$d\ge1$ by Lemma~\ref{lem:refinement-parameters}. Since
$\Phi_{0,a}(1)=1$ is independent of $a$, all partial
$\gamma$-polynomials agree, giving $P_\alpha(t)=P_\beta(t)$.

Otherwise, choose a new cut $c>1$ and let $d_0=c-1$.
Lemma~\ref{lem:refinement-parameters} gives
$1\le d_0\le N-2$ and $a_{d_0}(\beta)=1<a_{d_0}(\alpha)$.
By the positivity established after~\eqref{eq:gamma-coeff-op},
for $d\ge1$ and an intermediate polynomial
$\Gamma(z)=\sum_{i=0}^{\lfloor d/2\rfloor}\gamma_i z^i$, we have
\[
[z^i]\mathcal D_d(\Gamma(z))
=i\gamma_i+2(d-2i+2)\gamma_{i-1}>0
\qquad\left(1\le i\le\left\lfloor\frac{d+1}{2}\right\rfloor\right).
\]
For $1\le i\le\lfloor d/2\rfloor$, the first term is positive.
When $d$ is odd and $i=(d+1)/2$, the second term is
$2\gamma_{\lfloor d/2\rfloor}>0$.
It follows from~\eqref{eq:refinement-difference} that every
nonconstant coefficient of $\Delta_{d_0}(z)$ is positive.
At each later step, the positive multiplier of the existing coefficient
in~\eqref{eq:gamma-coeff-op} preserves this strict inequality.
When the degree increases, the new leading coefficient of the
difference is positive because the preceding leading coefficient
has multiplier $2/b>0$, while the second term
in~\eqref{eq:refinement-difference} remains nonnegative.
Since $d_0\ge1$, this proves the strict inequalities and the
equality classification.
\end{proof}

We now combine the coefficient and zero comparisons to prove
Theorem~\ref{thm:refinement}.

\begin{proof}[Proof of Theorem~\ref{thm:refinement}]
Lemma~\ref{lem:refinement-coefficients} gives the coefficient assertions
and reduces the equality classification to the cut sets.
The cut sets agree precisely when $\alpha=\beta$, and their difference is $\{1\}$
precisely when replacing the first part $r\ge2$ of $\alpha$ with $(1,r-1)$
gives $\beta$.

For the zero assertions, it suffices to consider $P_\alpha(t)\ne P_\beta(t)$.
Choose a new cut $c>1$ and put
$d_0=c-1$ as in Lemma~\ref{lem:refinement-parameters}.
At $d=0$ both partial $\gamma$-polynomials equal $1$.
Suppose $d\ge1$ and the zeros of
$\Gamma_\alpha^{[d-1]}(z)$ and $\Gamma_\beta^{[d-1]}(z)$ are
coordinatewise ordered. Since $\Phi_{d,a}$ is linear, we can normalize
both input polynomials to be monic without changing the zeros of their images.
Proposition~\ref{prop:gamma-root-monotonicity}, together with
$a_d(\beta)\le a_d(\alpha)$, therefore gives the same order for
the zeros of $\Gamma_\alpha^{[d]}(z)$ and $\Gamma_\beta^{[d]}(z)$.
Induction establishes the weak comparison at every step.

At step $d_0$, let $\boldsymbol\xi^\alpha$ and
$\boldsymbol\xi^\beta$ be the ordered vectors of zeros of
$\Gamma_\alpha^{[d_0-1]}(z)$ and $\Gamma_\beta^{[d_0-1]}(z)$,
respectively, and put $a=a_{d_0}(\alpha)$ and
$b=a_{d_0}(\beta)$. The same proposition gives
\[
\eta_i(\boldsymbol\xi^\alpha,a)
\le\eta_i(\boldsymbol\xi^\beta,a)
<\eta_i(\boldsymbol\xi^\beta,b)
\]
for every $i$. The first inequality uses the coordinatewise comparison,
and the second uses $b<a$. The same proposition propagates strictness
through every later step.
\end{proof}

The equality classification requires the refinement hypothesis.
For example, the incomparable compositions $(3,2)$ and $(4,1)$ both give
$P(t)=1+16t+36t^2+16t^3+t^4$.

Under a proper refinement outside the equality case, each
ordered zero in $(-\infty,-1)$ strictly decreases, whereas each ordered
zero in $(-1,0)$ strictly increases. Indeed, the derivative of
$t/(1+t)^2$ is $(1-t)/(1+t)^3$. This map is strictly decreasing on the
first interval and strictly increasing on the second, and maps each interval
bijectively onto $(-\infty,0)$. Each negative $\gamma$-zero has
one preimage in each interval. The left preimages occur in reverse
order and the right preimages in the same order, giving the assertion. When $N$ is even, the
common simple zero $-1$ remains fixed. For example,
$P_{(3)}(t)=1+3t+t^2$ and $P_{(2,1)}(t)=1+4t+t^2$ have nested rather
than interlacing pairs of zeros.

\subsection{Bounds from Narayana and Eulerian polynomials}

We apply Theorem~\ref{thm:refinement} to the coarsest and finest
compositions of $N$ to obtain these bounds.
Recall that
\[
P_M(t)=\sum_{k=1}^{N} c_{M,k}t^{k-1},
\qquad
c_{M,k}=\bigl|\{T\in\T_M:\Int(T)=k\}\bigr|,
\]
and put $q=\lfloor(N-1)/2\rfloor$. Write $\nu_i$, $\xi_i(M)$, and
$\varepsilon_i$ for the $i$th zeros, in increasing order, of
$\Gamma_{(N)}(z)$, $\Gamma_M(z)$, and $\Gamma_{\mathbf{1}_N}(z)$, respectively.

\begin{corollary}\label{cor:sandwich}
Let $M=\{1^{p_1},\ldots,m^{p_m}\}$ have total size $N\ge2$.  Then
\begin{gather*}
\mathrm{Nar}_N(t)\preceq_\gamma P_M(t)\preceq_\gamma \mathrm{Eul}_N(t),\\
\nu_i\le\xi_i(M)\le\varepsilon_i\qquad\text{for }1\le i\le q.
\end{gather*}
Moreover, $P_M(t)=\mathrm{Nar}_N(t)$ if and only if $(p_1,\ldots,p_m)$ is
$(N)$ or $(1,N-1)$, while $P_M(t)=\mathrm{Eul}_N(t)$ if and only if it is
$\mathbf{1}_N$ or $(2,\mathbf{1}_{N-2})$.  Outside the corresponding equality cases,
every nonconstant $\gamma$-coefficient inequality and every zero
inequality on that side is strict.
\end{corollary}

\begin{proof}
The identities $P_{(N)}(t)=\mathrm{Nar}_N(t)$ and $P_{\mathbf{1}_N}(t)=\mathrm{Eul}_N(t)$
give the bounds by Theorem~\ref{thm:refinement}. Put
$\alpha=(p_1,\ldots,p_m)$. By Lemma~\ref{lem:refinement-coefficients},
equality in the lower bound holds
precisely when $\operatorname{Cut}(\alpha)$ is empty or equals
$\{1\}$, giving $(N)$ or $(1,N-1)$.
For the upper bound, equality holds precisely when
$\{1,\ldots,N-1\}\setminus\operatorname{Cut}(\alpha)$ is empty
or equals $\{1\}$, giving $\mathbf{1}_N$ or $(2,\mathbf{1}_{N-2})$.
Every other case introduces a cut greater than $1$, giving strictness
by Theorem~\ref{thm:refinement}.
\end{proof}

The same bounds hold for ordinary coefficients because every basis
element $t^i(1+t)^{N-1-2i}$ has nonnegative coefficients.  Explicitly,
\[
\frac1N\binom Nj\binom N{j+1}
\le c_{M,j+1}\le
\left\langle\begin{matrix}N\\ j\end{matrix}\right\rangle
\qquad(0\le j\le N-1).
\]
For $N\ge3$, outside the corresponding equality cases, the first
nonconstant $\gamma$-coefficient inequality is strict, and
$t(1+t)^{N-3}$ has positive coefficients at every interior degree.
Thus the ordinary coefficient inequalities on that side are strict
for $1\le j\le N-2$.

\begin{remark}\label{rem:PF}
The coefficient sequences of $P_M(t)$ and $\Gamma_M(z)$ are
finite P\'olya frequency sequences of infinite order, also called
$\mathrm{PF}_\infty$ sequences. Equivalently, their infinite Toeplitz matrices,
with zeros outside the finite support, are totally nonnegative.
This follows from their negative zeros and the
Aissen--Schoenberg--Whitney theorem~\cite{ASW52}.
For background on these sequences, see~\cite{Brenti89,WangYeh05}.
Newton's inequalities also imply strict log-concavity of both
coefficient sequences at every interior index.
\end{remark}

Table~\ref{tab:small-compositions} illustrates the coefficient bounds and
equality cases for all compositions of size at most $4$.
For each composition, direct tree enumeration and the factorization in
Proposition~\ref{prop:global-factorization} give the same polynomial.
The last column gives the total number of trees.

\begin{table}[htbp]
\centering
\caption{The polynomials for all compositions of $N\le4$.}
\label{tab:small-compositions}
\small\begin{tabular}{@{}ccllc@{}}
\toprule
$N$ & $\alpha$ & $P_\alpha(t)$ & $\Gamma_\alpha(z)$ & $P_\alpha(1)$ \\
\hline
1 & $(1)$ & $1$ & $1$ & 1 \\
2 & $(1,1)$ & $1+t$ & $1$ & 2 \\
2 & $(2)$ & $1+t$ & $1$ & 2 \\
3 & $(1,1,1)$ & $1+4t+t^{2}$ & $1+2z$ & 6 \\
3 & $(1,2)$ & $1+3t+t^{2}$ & $1+z$ & 5 \\
3 & $(2,1)$ & $1+4t+t^{2}$ & $1+2z$ & 6 \\
3 & $(3)$ & $1+3t+t^{2}$ & $1+z$ & 5 \\
4 & $(1,1,1,1)$ & $1+11t+11t^{2}+t^{3}$ & $1+8z$ & 24 \\
4 & $(1,1,2)$ & $1+8t+8t^{2}+t^{3}$ & $1+5z$ & 18 \\
4 & $(1,2,1)$ & $1+9t+9t^{2}+t^{3}$ & $1+6z$ & 20 \\
4 & $(1,3)$ & $1+6t+6t^{2}+t^{3}$ & $1+3z$ & 14 \\
4 & $(2,1,1)$ & $1+11t+11t^{2}+t^{3}$ & $1+8z$ & 24 \\
4 & $(2,2)$ & $1+8t+8t^{2}+t^{3}$ & $1+5z$ & 18 \\
4 & $(3,1)$ & $1+9t+9t^{2}+t^{3}$ & $1+6z$ & 20 \\
4 & $(4)$ & $1+6t+6t^{2}+t^{3}$ & $1+3z$ & 14 \\
\hline
\end{tabular}
\end{table}

\FloatBarrier

\section{Transition matrices and column interlacing}\label{sec:transitions}

In this section, we study how maximal-label insertion changes the
number of internal nodes. We identify the resulting transition matrix with the
coefficient matrix of the insertion operator and use its bidiagonal
factorization to prove total nonnegativity and characterize the
positive minors. We then prove that adjacent column polynomials have
strictly interlacing zeros and derive a finite-difference identity
for these polynomials.

These transition matrices connect the forest decomposition with total positivity
in enumerative combinatorics~\cite{Brenti95}. Related work includes
criteria for recursive matrices~\cite{ChenLiangWang15}, total positivity
of Narayana matrices~\cite{WangYang18}, and nonnegativity of all minors of the
refined Eulerian matrix arising from barycentric subdivision~\cite{MuWangZhu26}.

\subsection{Transition coefficients}

Throughout this section, $\binom{a}{b}=0$ when $b<0$ or $b>a$.
Fix a tree $T\in\T_M$ with $|M|=p$ and $\Int(T)=l$, so that $T$ has
$L=p+1-l$ leaves.  We weight the attached maximal-label forests by their
contribution to the number of internal nodes, using $y$ to mark new
nodes and $t$ to mark internal nodes among them.
In a nonempty $a$-tree with $r$ nodes, the numbers of internal nodes
and leaves sum to $r$.  Thus weighting the series from
Section~\ref{sec:forest} by internal nodes gives
\[
\tau(y,t)=\mathcal U(yt,t^{-1}).
\]
Equation~\eqref{eq:pure-tree-leaf} becomes
\begin{equation*}
\tau(y,t)=y\left(1+t\frac{\tau(y,t)}{1-\tau(y,t)}\right).
\end{equation*}
Accordingly, define
\[
\Omega(y,t)=\frac1{1-\tau(y,t)}
=\mathcal F(yt,t^{-1}),
\qquad
\Theta(y,t)=\frac{\tau(y,t)}{y}
=1+t\bigl(\Omega(y,t)-1\bigr).
\]
At an internal node of the base tree, the attached forest contributes
$\Omega(y,t)$.  At a leaf, the empty forest contributes $1$, while a
nonempty forest contributes $t(\Omega(y,t)-1)$, where the extra $t$ records that
the base leaf becomes internal.  Including the $l$ internal nodes of
$T$, the generating function for all attachments is
\begin{equation}\label{eq:forest-gf}
\mathcal Z_{p,l}(y,t)=t^l\Omega(y,t)^l\Theta(y,t)^L.
\end{equation}
It depends only on $(p,l)$, not on the shape of $T$.  For $r\ge0$ and
$0\le s\le r$, define
\begin{equation*}
K_{l+s,l}^{(p,r)}=[y^rt^{l+s}]\mathcal Z_{p,l}(y,t),
\end{equation*}
and set $K_{k,l}^{(p,r)}=0$ outside the range $l\le k\le l+r$.  The
maximal-label forest decomposition gives the transition relation
\begin{equation*}
c_{M^{(r)},k}=\sum_{l=1}^p c_{M,l}K_{k,l}^{(p,r)}
\qquad(1\le k\le p+r).
\end{equation*}
Lagrange inversion gives a positive form of the coefficients in
Lin, Xu, and Zhao's recurrence~\cite[Theorem~4.1]{LXZ22}.

\begin{proposition}\label{prop:positive-transition}
Let $p\ge1$, $r\ge1$, $1\le l\le p$, and $L=p+1-l$.  For $0\le s\le r$,
\begin{equation}
K_{l+s,l}^{(p,r)}
=\binom{L+r}{s}\binom{l+r}{r-s}
\frac{l(L+r-s)+Ls}{(l+r)(L+r)}.
\label{eq:transition-positive}
\end{equation}
\end{proposition}

\begin{proof}
Put $\phi(u)=1+tu/(1-u)$ and
$\mathcal V(u)=(1-u)^{-l}\phi(u)^L$. Then
$\tau(y,t)=y\phi(\tau(y,t))$ and
$t^{-l}\mathcal Z_{p,l}(y,t)=\mathcal V(\tau(y,t))$.
Lagrange inversion gives
\[
[y^r]\mathcal V(\tau(y,t))
=\frac1r[u^{r-1}]\left(
l(1-u)^{-l-1}\phi(u)^{L+r}
+Lt(1-u)^{-l-2}\phi(u)^{L+r-1}
\right).
\]
Expanding the powers of $\phi(u)=1+tu/(1-u)$ and extracting the
coefficient of $t^s$ yields
\begin{equation}\label{eq:transition-sum}
\begin{aligned}
K_{l+s,l}^{(p,r)}
={}&\frac lr\binom{L+r}{s}\binom{l+r-1}{r-s-1}\\
&+\frac Lr\binom{L+r-1}{s-1}\binom{l+r}{r-s}.
\end{aligned}
\end{equation}
The zero convention covers $s=0$ and $s=r$. Factor out
$\binom{L+r}{s}\binom{l+r}{r-s}$, using
\[
\binom{l+r-1}{r-s-1}=\frac{r-s}{l+r}\binom{l+r}{r-s},
\qquad
\binom{L+r-1}{s-1}=\frac{s}{L+r}\binom{L+r}{s}.
\]
The remaining factor simplifies as
\[
\frac{l(r-s)}{r(l+r)}+\frac{Ls}{r(L+r)}
=\frac{l(L+r-s)+Ls}{(l+r)(L+r)},
\]
which gives~\eqref{eq:transition-positive}.
\end{proof}

\begin{example}\label{ex:transition-matrix}
For $p=3$ and $r=2$, we index the rows of $K^{(3,2)}$ by
$k=1,\ldots,5$ and its columns by $l=1,2,3$. Formula~\eqref{eq:transition-positive}
gives
\[
K^{(3,2)}=\begin{pmatrix}1 & 0 & 0\\7 & 3 & 0\\6 & 8 & 6\\0 & 3 & 7\\0 & 0 & 1\end{pmatrix}.
\]
The minor on rows $2,3$ and columns $1,2$ is
$7\cdot8-3\cdot6=38$. Each column counts extensions of a fixed base tree.
\end{example}

Setting $p_n=r$ and $p'=p+r$ in the recurrence of
Lin, Xu, and Zhao~\cite[Theorem~4.1]{LXZ22}, we obtain
\begin{align*}
K_{k,l}^{(p,r)}
={}&\binom{r+l-1}{k-1}\binom{p+r+1-l}{k-l}
\\
&-\binom{r+l}{k}\binom{p+r-l}{k-l-1}.
\end{align*}
Indeed, after setting $k=l+s$ and $L=p+1-l$, the right-hand side factors as
\[
\binom{l+r}{r-s}\binom{L+r}{s}
\left(\frac{l+s}{l+r}-\frac{s}{L+r}\right),
\]
which equals~\eqref{eq:transition-positive}. Every entry in the
band $l\le k\le l+r$ is positive.

We next identify the transition matrix with the coefficient matrix
of the block operator. Let $\mathcal B_{p,r}$ denote the operator
in~\eqref{eq:block-univ-r1} for $r=1$ and in~\eqref{eq:block-univ}
for $r\ge2$, acting on polynomials of degree at most $p-1$.
Its action on a monomial is as follows.

\begin{lemma}\label{lem:operator-transition-matrix}
For $p\ge1$, $r\ge1$, and $1\le l\le p$,
\begin{equation*}
\mathcal B_{p,r}(t^{l-1})
=\sum_{s=0}^r K_{l+s,l}^{(p,r)}t^{l+s-1}.
\end{equation*}
Consequently, the coefficient matrix of $\mathcal B_{p,r}$ in the monomial
bases $(1,t,\ldots,t^{p-1})$ and $(1,t,\ldots,t^{p+r-1})$ is
$K^{(p,r)}$.
\end{lemma}

\begin{proof}
Put $U=2x-1$, $V=2x+1$, and $L=p+1-l$.  For $f(t)=t^{l-1}$,
\[
Q(x)\mathscr C_{p-1}(f(t))(x)
=2^{-(p+1)}U^L V^l.
\]
Repeated use of Lemma~\ref{lem:A-E-conjugacy}, together with
Proposition~\ref{prop:differential-factorization}, gives
\begin{equation*}
Q(x)\mathscr C_{p+r-1}\bigl(\mathcal B_{p,r}(f(t))\bigr)(x)
=
\frac1{r!}\left(\frac d{dx}\right)^{r-1}
\left(
Q(x)^r\frac d{dx}
\bigl(Q(x)\mathscr C_{p-1}(f(t))(x)\bigr)
\right).
\end{equation*}
For $0\le s\le r$,
\[
Q(x)\mathscr C_{p+r-1}(t^{l+s-1})(x)
=2^{-(p+r+1)}U^{L+r-s}V^{l+s}.
\]
The inner derivative differentiates either $U^L$ or $V^l$.
After multiplication by $Q(x)^r$, obtaining
$U^{L+r-s}V^{l+s}$ requires the remaining $r-1$ derivatives
to act on the power of $U$ exactly $s-1$ or $s$ times,
respectively. A choice outside $0,\ldots,r-1$ contributes zero.
Since $U'=V'=2$ and $Q(x)=UV/4$, the powers of $2$ combine
to $2^{-(p+r+1)}$, matching the transformed monomial above.
Leibniz's rule therefore gives the right-hand side of
\eqref{eq:transition-sum} as the coefficient of $t^{l+s-1}$
in $\mathcal B_{p,r}(t^{l-1})$, proving that this coefficient
is $K_{l+s,l}^{(p,r)}$.
Since~\eqref{eq:E-def} only preserves or raises a monomial exponent
by one, all other coefficients vanish, which completes the proof.
\end{proof}

\subsection{Total nonnegativity}

We first prove total nonnegativity using the bidiagonal factorization
and then determine which minors are positive.
See~\cite[Chapter~2]{FallatJohnson11} and~\cite{Pinkus10} for background on such factorizations.
Let $f(t)=\sum_{j=0}^dc_jt^j$ and set $c_j=0$ for
$j\notin\{0,1,\ldots,d\}$.  Then, for $0\le j\le d+1$,
\begin{equation*}
[t^j]\E_{d,a}(f(t))
=\frac{j+a}{a}c_j
+\frac{d+a-j+1}{a}c_{j-1}.
\end{equation*}
Thus, in the monomial bases $(1,t,\ldots,t^d)$ and
$(1,t,\ldots,t^{d+1})$, the coefficient matrix of $\E_{d,a}$ is a
rectangular lower-bidiagonal matrix with positive entries on its
two diagonals. For example, at $d=2$ it is
\[
\frac1a\begin{pmatrix}
a&0&0\\
a+2&a+1&0\\
0&a+1&a+2\\
0&0&a
\end{pmatrix}.
\]

We now use the bidiagonal factorization to prove Theorem~\ref{thm:intro-TN}.

\begin{proof}[Proof of Theorem~\ref{thm:intro-TN}]
By Lemma~\ref{lem:operator-transition-matrix}, $K^{(p,r)}$ is the
product of these bidiagonal matrices. Each factor has positive entries
in column $j$ only in rows $j$ and $j+1$. Hence the bipartite
support graph of any square submatrix is a subgraph of a path and has at
most one perfect matching.  If such a matching exists, it is
order-preserving, so the determinant is the product of the matched positive
entries.  Otherwise the determinant is zero, so every first-order factor
is totally nonnegative. Applying Cauchy--Binet to their product, we conclude that
$K^{(p,r)}$ is totally nonnegative.
\end{proof}

\begin{proposition}\label{prop:positive-minors}
For $p,r\ge1$, a nonempty minor of $K^{(p,r)}$ on rows
$k_1<\cdots<k_h$ and columns $l_1<\cdots<l_h$ is positive if and only if
\[
l_i\le k_i\le l_i+r\qquad\text{for }1\le i\le h.
\]
\end{proposition}

\begin{proof}
The band support
$K_{k,l}^{(p,r)}=0$ unless $l\le k\le l+r$ gives the necessity of the
stated inequalities.  Indeed, if $k_i<l_i$ for some $i$, then
the first $i$ selected rows vanish on the last $h-i+1$ selected columns,
so these rows have support only in the first $i-1$ columns and are
linearly dependent.  If $k_i>l_i+r$, then the last $h-i+1$ selected rows
vanish on the first $i$ selected columns and have support only in the last
$h-i$ columns.  In either case the minor is zero.

Conversely, assume that $l_i\le k_i\le l_i+r$ for every $i$.
We construct a positive term in the Cauchy--Binet expansion of the minor.
List the $r$ first-order factors in the order in which they act, and let
$B_s$ be the matrix of the $s$th factor.  We index its columns by
$\{1,\ldots,p+s-1\}$ and its rows by $\{1,\ldots,p+s\}$. Its positive
entries occur precisely in positions $(j,j)$ and $(j+1,j)$.  Thus
$K^{(p,r)}=B_r\cdots B_1$.  For $0\le s\le r$, put
\[
u_i(s)=\min(l_i+s,k_i),
\qquad I_s=\{u_i(s):1\le i\le h\}.
\]
We have $1\le u_i(s)\le l_i+s\le p+s$. Both selected index
sequences are strictly increasing, and for $i<h$,
\[
u_i(s)\le l_i+s<l_{i+1}+s,
\qquad
u_i(s)\le k_i<k_{i+1}.
\]
Hence $u_i(s)<\min(l_{i+1}+s,k_{i+1})=u_{i+1}(s)$.
The band inequalities give $u_i(0)=l_i$ and $u_i(r)=k_i$.
For $1\le s\le r$, let $B_s[I_s,I_{s-1}]$ denote
the submatrix of $B_s$ with row indices $I_s$ and column indices
$I_{s-1}$, both listed in increasing order.
Since $u_i(s)-u_i(s-1)$ is either $0$ or $1$, pairing column
$u_i(s-1)$ with row $u_i(s)$ gives an
order-preserving matching on positive entries in $B_s[I_s,I_{s-1}]$.
By the support-graph argument above, this is its unique perfect
matching. Thus the determinant expansion has exactly one nonzero
term, whose sign is positive because the matching preserves order.
The determinant is therefore the product of positive matched entries.

In the iterated Cauchy--Binet expansion of the chosen minor of
$B_r\cdots B_1$, the chain
$I_0,I_1,\ldots,I_r$ therefore contributes a strictly positive term, while
all other terms are nonnegative.  Since $I_0=\{l_1,\ldots,l_h\}$ and
$I_r=\{k_1,\ldots,k_h\}$, the desired minor of $K^{(p,r)}$ is strictly
positive.
\end{proof}

In particular, for $l_1<l_2$, the nonnegativity of the $2\times2$ minors implies
that $K_{k,l_2}^{(p,r)}/K_{k,l_1}^{(p,r)}$ is nondecreasing in $k$
on the common positive support of the two columns.
Since products of totally nonnegative matrices remain totally
nonnegative, these ratio comparisons persist under successive block
insertions.

\subsection{Column polynomials}

For a fixed initial internal-node count, the entries of a column
enumerate the possible increases in this count. For the $l$th
column, define
\begin{equation*}
R_{p,l,r}(t)=\sum_{s=0}^rK_{l+s,l}^{(p,r)}t^s,
\qquad r\ge1,
\end{equation*}
and adopt the convention $R_{p,l,0}(t)=1$.
Here the base tree has $p+1$ nodes, of which $l$ are internal.
The endpoint coefficients and the value at $1$ are
\begin{equation*}
\begin{aligned}
R_{p,l,r}(0)&=\binom{l+r-1}{r},
&[t^r]R_{p,l,r}(t)&=\binom{L+r-1}{r},\\
R_{p,l,r}(1)&=\frac{p+1}{p+r+1}\binom{p+2r}{r},
&& L=p+1-l.
\end{aligned}
\end{equation*}
The endpoint formulas follow from~\eqref{eq:transition-positive}.
For the last formula, set $t=1$ in~\eqref{eq:forest-gf}.  Then
$\Omega(y,1)=\Theta(y,1)=C(y)$, where $C(y)=1+yC(y)^2$, and the forest
series is $C(y)^{p+1}$. Lagrange inversion gives
\[
R_{p,l,r}(1)=[y^r]C(y)^{p+1}
=\frac{p+1}{p+2r+1}\binom{p+2r+1}{r}
=\frac{p+1}{p+r+1}\binom{p+2r}{r},
\]
which is independent of $l$. These identities also hold for $r=0$.

Interchanging $l$ and $L$ and replacing $s$ by $r-s$ in
\eqref{eq:transition-positive} gives the symmetry
\[
R_{p,l,r}(t)=t^rR_{p,p+1-l,r}(1/t).
\]

\begin{proposition}\label{prop:column-rooted}
For $p,r\ge1$ and $1\le l\le p$, the polynomial $R_{p,l,r}(t)$ has
$r$ simple negative zeros.
\end{proposition}

\begin{proof}
By Lemma~\ref{lem:operator-transition-matrix} and the definition of
$R_{p,l,r}(t)$,
\[
\mathcal B_{p,r}(t^{l-1})=t^{l-1}R_{p,l,r}(t).
\]
At each step, write the image as $t^{l-1}h(t)$ and let $d_*$ be the
current degree parameter.  The difference
$d_*-(l-1+\deg h(t))=p-l\ge0$ stays constant along the factorization.
By~\eqref{eq:E-root-action}, Lemma~\ref{lem:E-root-insertion}
applies at all $r$ factors,
starting with $h(t)=1$, and gives the asserted zeros.
\end{proof}

\subsection{Interlacing and compatibility of adjacent columns}

We compare adjacent columns through the action of the block operator
on a linear combination of two consecutive monomials.
\begin{lemma}\label{lem:column-pencil-preservation}
Put $\zeta_\pm=\pm\tfrac12$.
\begin{enumerate}
\item[(i)] Let $g(x)\ne0$ be real-rooted with simple zeros outside
$\{\zeta_-,\zeta_+\}$, and let $a\ge1$ be an integer.
Then $\mathcal A_a(g(x))$ has the same properties, has degree
$\deg g(x)+1$, and preserves the multiplicity at each endpoint.

\item[(ii)] For integers $0\le e\le d$ and a real polynomial $g(x)$
of degree $e$, the inverse image of $g(x)$ under $\mathscr C_d$ is
\[
\varphi(t)=(t-1)^d g\left(\frac{t+1}{2(t-1)}\right).
\]
If $g(x)$ has multiplicities $m_-$ and $m_+$ at $\zeta_-$ and $\zeta_+$,
respectively, then
\[
\deg\varphi(t)=d-m_+,
\qquad \operatorname{ord}_0\varphi(t)=m_-,
\qquad \operatorname{ord}_1\varphi(t)=d-e.
\]
Its remaining zeros are $(2\xi+1)/(2\xi-1)$, where $\xi$ runs
through the other zeros of $g(x)$, with the same multiplicity.
Thus this transformation preserves real-rootedness, and negative zeros of $\varphi(t)$
correspond to zeros of $g(x)$ in $(\zeta_-,\zeta_+)$.
\end{enumerate}
\end{lemma}

\begin{proof}
For (i), $W(x)=Q(x)^a g(x)$ is real-rooted, with possible multiple zeros
only at the endpoints. Lemma~\ref{lem:rolle-multiplicities} and
$W'(x)=Q(x)^{a-1}\mathcal A_a(g(x))$ give real-rootedness and simplicity
away from them. An endpoint of multiplicity $m$ in $g(x)$ has
multiplicity $m+a-1$ in $W'(x)$, leaving $m$ after division by
$Q(x)^{a-1}$. If $c$ is the leading coefficient of $g(x)$, the leading
term of $\mathcal A_a(g(x))$ is $(\deg g(x)+2a)cx^{\deg g(x)+1}$.

For (ii), the substitution $x=(t+1)/(2(t-1))$ inverts
$t=(2x+1)/(2x-1)$. Let $c$ be the leading coefficient of $g(x)$,
and list its $s=e-m_--m_+$ nonendpoint zeros as $\xi_1,\ldots,\xi_s$,
with multiplicity. Substitution in its factorization gives
\[
\varphi(t)=\frac{c}{2^s}(t-1)^{d-e}t^{m_-}
\prod_{j=1}^{s}\bigl((1-2\xi_j)t+1+2\xi_j\bigr).
\]
Each linear factor has nonzero leading and constant coefficients
and value $2$ at $t=1$. This proves the degree and multiplicity
claims. For real $\xi$, its image is negative exactly when
$-\tfrac12<\xi<\tfrac12$.
\end{proof}

We now complete the proof of Theorem~\ref{thm:intro-columns} by
determining the order of the interlacing zeros of adjacent column
polynomials. For $p\ge2$, $r\ge1$, and $1\le l\le p$, write
\[
\rho_1^{(l)}<\cdots<\rho_r^{(l)}<0
\]
for the zeros of $R_{p,l,r}(t)$.

\begin{theorem}\label{thm:strict-columns}
For $p\ge2$, $r\ge1$, and $1\le l<p$,
\[
\rho_1^{(l+1)}
<\rho_1^{(l)}
<\rho_2^{(l+1)}
<\rho_2^{(l)}
<\cdots
<\rho_r^{(l+1)}
<\rho_r^{(l)}<0.
\]
\end{theorem}

\begin{proof}
Fix $p\ge2$, $r\ge1$, and $1\le l<p$, and put
\[
f(t)=R_{p,l,r}(t),\qquad h(t)=tR_{p,l+1,r}(t).
\]
We apply the Obreschkoff--Dedieu criterion by showing that every
nontrivial real linear combination of $f(t)$ and $h(t)$ has only simple
real zeros.
For real $(\lambda,\mu)\ne(0,0)$, let
$\psi(t)=t^{l-1}(\lambda+\mu t)$.
Lemma~\ref{lem:operator-transition-matrix} gives
\[
\mathcal B_{p,r}(\psi(t))
=t^{l-1}\bigl(\lambda f(t)+\mu h(t)\bigr).
\]
With $U=2x-1$ and $V=2x+1$, we have
\[
\mathscr C_{p-1}(\psi(t))(x)
=2^{-(p-1)}U^{p-l-1}V^{l-1}(\lambda U+\mu V).
\]
This is nonzero and real-rooted, with at most one zero outside
$\{-\tfrac12,\tfrac12\}$, which is simple if present.

Set
\[
g(x)=\mathscr C_{p+r-1}\bigl(\mathcal B_{p,r}(\psi(t))\bigr)(x).
\]
The endpoint multiplicities of $\mathscr C_{p-1}(\psi(t))(x)$ at
$\zeta_-=-\tfrac12$ and $\zeta_+=\tfrac12$ are
\[
m_-=l-1+\mathbf 1_{\{\lambda=0\}},
\qquad
m_+=p-l-1+\mathbf 1_{\{\mu=0\}},
\]
and
\[
\deg \mathscr C_{p-1}(\psi(t))(x)
=p-1-\mathbf 1_{\{\lambda+\mu=0\}}.
\]
Indeed, the last factor $\lambda U+\mu V$ vanishes at
$\zeta_-$ exactly when $\lambda=0$, at $\zeta_+$ exactly when
$\mu=0$, and becomes the nonzero constant $-2\lambda$ exactly when
$\lambda+\mu=0$.

The block operator contains $r$ first-order factors. By the conjugacy in
Lemma~\ref{lem:A-E-conjugacy} and
Lemma~\ref{lem:column-pencil-preservation}(i), these factors preserve
the two endpoint multiplicities, increase the degree by $r$, and leave
all nonendpoint zeros real and simple.  Hence $g(x)$ has endpoint
multiplicities $m_-$ and $m_+$, all its other zeros are simple and real,
and
\[
\deg g(x)=p+r-1-\mathbf 1_{\{\lambda+\mu=0\}}.
\]
Applying Lemma~\ref{lem:column-pencil-preservation}(ii), with
$d=p+r-1$, gives
\[
\operatorname{ord}_0\mathcal B_{p,r}(\psi(t))=m_-,
\qquad
\operatorname{ord}_1\mathcal B_{p,r}(\psi(t))
=\mathbf 1_{\{\lambda+\mu=0\}},
\]
and every remaining zero of $\mathcal B_{p,r}(\psi(t))$ is the
M\"obius image of a nonendpoint zero of $g(x)$ with the same
multiplicity.  Since
\[
\mathcal B_{p,r}(\psi(t))
=t^{l-1}\bigl(\lambda f(t)+\mu h(t)\bigr),
\]
we obtain
\[
\operatorname{ord}_0(\lambda f(t)+\mu h(t))
=\mathbf 1_{\{\lambda=0\}},
\qquad
\operatorname{ord}_1(\lambda f(t)+\mu h(t))
=\mathbf 1_{\{\lambda+\mu=0\}}.
\]
Thus every nonzero polynomial in the real pencil
$\{\lambda f(t)+\mu h(t):\lambda,\mu\in\mathbb R\}$ has only simple
real zeros. These zeros need not all be negative, since
$\lambda f(t)+\mu h(t)$ has a simple zero at $1$ when $\lambda+\mu=0$.

Since $\deg h(t)=\deg f(t)+1$, we conclude from the
Obreschkoff--Dedieu criterion in the form stated by
Br\"and\'en~\cite[Theorem~8]{Branden04} that $f(t)$ strictly
interlaces $h(t)$.
The zeros of $h(t)$ are $0$ and the negative zeros of
$R_{p,l+1,r}(t)$, so we obtain the stated order and complete
the proof of Theorem~\ref{thm:intro-columns}.
\end{proof}

Following Chudnovsky and Seymour~\cite{CS07}, two real polynomials are
\emph{compatible} if every nontrivial nonnegative linear combination
is real-rooted. In particular, adjacent column polynomials are
compatible, and Theorem~\ref{thm:strict-columns} also locates the zeros
of their nonnegative linear combinations. For
$\lambda,\mu\ge0$ with $\lambda+\mu>0$, the polynomial
$\lambda R_{p,l,r}(t)+\mu R_{p,l+1,r}(t)$ has $r$ simple negative zeros
$\eta_1<\cdots<\eta_r$, and if $\lambda,\mu>0$, then
\[
\rho_i^{(l+1)}<\eta_i<\rho_i^{(l)}\qquad(1\le i\le r).
\]
For positive $\lambda,\mu$, the combination has opposite signs at the
endpoints of each of these $r$ intervals, so degree counting gives the
assertion.  If one coefficient is zero, it follows from
Proposition~\ref{prop:column-rooted}.

\begin{samepage}
We need the restriction to adjacent columns. For example,
\[
R_{3,1,3}(t)=1+12t+25t^2+10t^3,
\qquad
R_{3,3,3}(t)=10+25t+12t^2+t^3,
\]
but
\[
2R_{3,1,3}(t)+R_{3,3,3}(t)
=12+49t+62t^2+21t^3
\]
has discriminant $-27056$, so it is not real-rooted.
\end{samepage}

The same forest generating function determines the zeros of the finite
differences in the column index.

\begin{proposition}\label{prop:column-differences}
Let $\Delta_l$ denote the forward difference in $l$.
For $p,r\ge1$, $0\le q\le r$, and $1\le l\le p-q$,
\[
\Delta_l^qR_{p,l,r}(t)
=(1-t)^qR_{p+q,l+q,r-q}(t).
\]
The polynomial on the left has $r-q$ simple negative zeros and
vanishing order $q$ at $1$.
\end{proposition}

\begin{proof}
We set $\Psi(y,t)=\Omega(y,t)/\Theta(y,t)$.
Equation~\eqref{eq:forest-gf} gives
\[
\sum_{r\ge0}R_{p,l,r}(t)y^r=\Theta(y,t)^{p+1}\Psi(y,t)^l.
\]
Since $\Theta(y,t)=1+t(\Omega(y,t)-1)$ and
$\Omega(y,t)-1=y\Omega(y,t)\Theta(y,t)$, we obtain
\begin{align*}
\Psi(y,t)-1
&=\frac{\Omega(y,t)-\Theta(y,t)}{\Theta(y,t)}
=\frac{(1-t)(\Omega(y,t)-1)}{\Theta(y,t)}\\
&=(1-t)y\Omega(y,t)=(1-t)y\Theta(y,t)\Psi(y,t).
\end{align*}
Taking $q$ forward differences gives
\[
\Delta_l^q\bigl(\Theta(y,t)^{p+1}\Psi(y,t)^l\bigr)
=(1-t)^qy^q\Theta(y,t)^{p+q+1}\Psi(y,t)^{l+q}.
\]
Extracting $[y^r]$ proves the stated identity.
The zero assertions follow from Proposition~\ref{prop:column-rooted}
and $R_{p+r,l+r,0}(t)=1$, since the remaining factor is positive at $1$.
\end{proof}

\section{Further remarks}\label{sec:final-remarks}

In this section, we compare successive applications of the first-order
factors with refinement. A first-order factor increases the degree and inserts
zeros, whereas refinement fixes the degree and moves reciprocal
pairs weakly away from $-1$. We explain the limitations of these zero
comparisons and of strict positivity for the transition matrices
and then discuss the related real-rootedness conjecture for reduced
multiset Schett polynomials.

Distinct polynomials $P_\alpha(t)$ and
$P_\beta(t)$ of the same total size cannot have interlacing zeros.
To see this, let $x_1<\cdots<x_d<0$ and $y_1<\cdots<y_d<0$
be their zeros. Interlacing, in either orientation, would imply
$x_i\le y_i$ for every $i$ after interchanging the two polynomials
if necessary. Since both polynomials are monic with constant term
$1$,
\[
\prod_{i=1}^d(-x_i)=\prod_{i=1}^d(-y_i)=1.
\]
Since $-x_i\ge-y_i>0$, equality of the products forces
$x_i=y_i$ for every $i$, contradicting distinctness.

The band support also prevents $K^{(p,r)}$ from being
totally positive when $p\ge2$, since $K_{1,2}^{(p,r)}=0$.
Proposition~\ref{prop:positive-minors} identifies precisely its positive minors.

Lin and Ma~\cite[Conjecture~5.1]{LinMa25} posed a related
real-zero question for reduced multiset Schett polynomials.
Let $\mathsf{ee}(T)$ and $\mathsf{oe}(T)$ count the
nodes of even outdegree at even and odd levels, respectively,
with the root at level $0$. Their conjecture concerns the polynomials
\[
\widehat S_{M,i}(t)=
\sum_{\substack{T\in\T_M\\\lfloor\mathsf{ee}(T)/2\rfloor=i}}
t^{\lfloor\mathsf{oe}(T)/2\rfloor},
\]
and asserts that every nonzero such polynomial has only real zeros.
To apply maximal-label deletion to these slices, we would need to
record degree parity at each base vertex and level parity in every
attached forest. Our distinction between leaves and internal nodes
does not retain that information.
Finding a tractable insertion operator for these statistics is a
natural extension of the present approach.

\FloatBarrier
\section*{Acknowledgements}
Philip B.~Zhang received support from the National Natural Science
Foundation of China under Grant No.~12171362 and Natural Science
Foundation of Tianjin Municipality under Grant No.~\mbox{25JCYBJC00430}.
Tongyuan Zhao received support from the National Natural Science
Foundation of China under Grant No.~12201641.

\section*{Declaration of Generative AI}
During the preparation of this manuscript, the authors used generative AI
tools to assist in exploring possible approaches, checking technical details,
and improving the exposition. The authors verified the mathematical arguments, computations, proofs,
and results and take full
responsibility for the content of this manuscript.

\end{document}